\documentclass[11pt, amstex]{amsart}
\usepackage{calc,amssymb,amsmath,ulem}
\usepackage{alltt}
\RequirePackage[dvipsnames,usenames]{color}

\input{kmacros3.sty}
\input{xy}
\xyoption{all}
\usepackage{hyperref}
\usepackage{graphicx}
\usepackage[all,cmtip]{xy}
\usepackage{verbatim}
\usepackage[left=1.2in,top=1in,right=1.2in,bottom=1in]{geometry}
\renewcommand{\tauCohomology}{T}
\newcommand{\FCohomology}{S}

\renewcommand{\O}{\mathcal O}

\renewcommand{\to}[1][]{\xrightarrow{\ #1\ }}

\begin{document}
\numberwithin{equation}{theorem}
\title[A canonical linear system]{A canonical linear system associated to adjoint divisors in characteristic $p > 0$}
\author{Karl Schwede}
\begin{abstract}Suppose that $X$ is a projective variety over an algebraically closed field of characteristic $p > 0$.  Further suppose that $L$ is an ample (or more generally in some sense positive) divisor.  We study a natural linear system in $|K_X + L|$.  We further generalize this to incorporate a boundary divisor $\Delta$.
We show that these subsystems behave like the global sections associated to multiplier ideals, $H^0(X, \mJ(X, \Delta) \tensor L)$ in characteristic zero.  In particular, we show that these systems are in many cases base-point-free.  While the original proof utilized Kawamata-Viehweg vanishing and variants of multiplier ideals, our proof uses test ideals.
\end{abstract}
\thanks{The author was partially supported by an NSF postdoctoral fellowship \#0703505 and the NSF grant DMS \#1064485}

\address{Department of Mathematics\\ The Pennsylvania State University\\ University Park, PA, 16802, USA}
\email{schwede@math.psu.edu}
\maketitle

\section{Introduction}

Suppose that $X$ is a normal variety and $\Delta$ is an effective $\bQ$-divisor such that $K_X + \Delta$ is $\bQ$-Cartier.  In characteristic zero, to this data one associates an ideal $\mJ(X, \Delta)$ called the \emph{multiplier ideal}, which reflects subtle properties of both the singularities and $X$ and $\Delta$.  The utility of the multiplier ideal for projective varieties lies mostly in its connections with a form of Kawamata-Viehweg vanishing usually called Nadel vanishing.  For example, consider the following application of the aforementioned vanishing and Castelnuovo-Mumford regularity.  If $M$ is a globally generated ample Cartier divisor on an $n$-dimensional normal variety and $L$ is another Cartier divisor such that $L - K_X - \Delta$ is big and nef, then it follows that $\mJ(X, \Delta) \tensor \O_X(nM+ L)$ is globally generated, \cf \cite[Proposition 9.4.26]{LazarsfeldPositivity2}.

In characteristic $p > 0$, to the same data $(X, \Delta)$, one can associate an ideal $\tau(X, \Delta)$ called the \emph{test ideal}.  The test ideal agrees with the multiplier ideal after reduction to characteristic $p \gg 0$ and it has been shown that $\tau(X, \Delta)$ satisfies many of the same \emph{local} properties as the multiplier ideal.  However, global applications of the test ideal have so far been lacking.  This is largely because the Nadel-type vanishing theorems are known to be false.  In this paper we prove the following result.

\begin{theorem*} [Corollary \ref{cor.FullGlobalGeneration}]
Suppose that $X$ is an $n$-dimensional normal projective variety over an $F$-finite field and that $M$ is a globally generated ample Cartier divisor.  Let $\Delta$ be an effective $\bQ$-divisor such that $K_X + \Delta$ is $\bQ$-Cartier and suppose that $L$ is a Cartier divisor such that $L - K_X - \Delta$ is big and nef.

Then
\[
\tau(X, \Delta) \tensor \O_X(L + nM)
\]
is globally generated.  In particular, if $(X, \Delta)$ is strongly $F$-regular, then $\O_X(L + nM)$ is globally generated.
\end{theorem*}

The proof strategy is the same as D.~Keeler's simplified proof of some special cases of Fujita's conjecture in characteristic $p > 0$ \cite{KeelerFujita}.  In this special case of Fujita's conjecture, N.~Hara also independently observed the same proof at roughly the same time as Keeler, and while he gave public talks on the method, he did not write it down.  We will discuss this proof and generalizations in Theorem \ref{thm.KeelerThm}.

One application of the aforementioned global generation theorem for multiplier ideals is the construction of effective bounds on the degrees of hypersurfaces passing through points on $\mathbb{P}^n$.  More generally, in characteristic zero, these bounds are a special case of a result originally proved using vanishing theorems and multiplier-ideal-like constructions in \cite{EsnaultViehwegSurUneMinoration}, \cf \cite{WaldschmidtNombresTranscendants}.  Another proof of these bounds explicitly using the language of multiplier ideals can be found in \cite[Section 10.1]{LazarsfeldPositivity2}, \cf \cite{EsnaultViehwegLecturesOnVanishing}.  This latter proof goes through without change and so we obtain a proof of the following characteristic $p > 0$ theorem that exactly mimics global proofs using multiplier ideals.

\begin{theorem*}[Theorem \ref{thm.HypersurfaceDegreeApplication}] \cite[Slide 17]{HarbourneAsymptoticInvariantsSlides}, \cf \cite[Section 10.1]{LazarsfeldPositivity2}.

Fix a reduced closed subscheme $S \subseteq \bP^n_k = X$, where $k$ is an algebraically closed field of characteristic $p > 0$, such that each irreducible component $Z \subseteq  S$ has codimension $\leq e$ in $\bP^n_k$.
Suppose that $A \subseteq \bP^n_k$ is a hypersurface of degree $d$ such that $\mult_x A \geq l$ for all $x \in S$.  Then $S$ lies on a hypersurface of degree $\lfloor {de \over l} \rfloor$.
\end{theorem*}

This result is not new to experts in characteristic $p > 0$. Indeed, this bound has been recently observed by B.~Harbourne as a corollary of \cite{HochsterHunekeComparisonOfSymbolic}, see \cite[Slide 17]{HarbourneAsymptoticInvariantsSlides} and \cf \cite{HarbourneHunekeSymbolicPowersHighly,EinLazSmithSymbolic}.  However, it still demonstrates that global statements typically proven using multiplier ideals can be obtained through the use of test ideals.

We now leave our discussion of applications.
We actually prove a substantially stronger theorem than the mentioned global generation statement, Corollary \ref{cor.FullGlobalGeneration} above.  We identify a canonical vector sub-space of $H^0\left(X, \tau(X, \Delta) \tensor \O_X(L + nM)\right)$ which globally generates $\tau(X, \Delta) \tensor \O_X(L + nM)$.  It is the study of the associated linear system that inspires the title of this paper.  Let us briefly describe how this linear system arises in a simple case.  Suppose that $X$ is a variety over a perfect field $k$ and that $\sL = \O_X(L)$ is any line bundle (we are mostly interested in the case when $L$ is positive, {\itshape e.g.} ample).  Set $F^e : X \to X$ to be the absolute $e$-iterated Frobenius.  There is a natural $\O_X$-module map $F^e_* \omega_X \to \omega_X$ which arises from the Cartier operator if $X$ is smooth, but exists generally as the Grothendieck-trace map.  By twisting this map by $\sL$ and taking cohomology we obtain maps
\[
 H^0\left(X, F^e_* (\omega_X \tensor \sL^{p^e})\right) \to H^0\left(X, \omega_X \tensor \sL\right).
\]
For $e \gg 0$, the image of this map stabilizes (note $H^0\left(X, \omega_X \tensor \sL\right)$ is a finite dimensional vector space) and so induces a linear system.  We denote this stable image by $\FCohomology^0(X, \omega_X \tensor \sL)$.  We should note that this image is sometimes a proper submodule of $ H^0\left(X, \omega_X \tensor \sL\right)$, even if the map $H^0(X, F^e_* \omega_X) \to H^0(X, \omega_X)$ is bijective, see for example \cite{TangoOnTheBehaviorOfVBAndFrob}.

Suppose now that $X$ is $F$-rational and $\sL$ is a globally generated ample line bundle.  In \cite{SmithFujitasFreeness, HaraACharacteristicPAnalogOfMultiplierIdealsAndApplications, KeelerFujita} it is shown that $\omega_X \tensor \sL^{\dim X + 1}$ is globally generated, \cf \cite{ArapuraFAmplitudeAndStrongVanishing}.  However, essentially using D.~Keeler's proof (also independently observed by N.~Hara, as mentioned above), we actually conclude that $\omega_X \tensor \sL^{\dim X + 1}$ is globally generated by $\FCohomology^0(X, \omega_X \tensor \sL^{\dim X})$.  In fact, the same result holds under the weaker hypothesis that $X$ has $F$-injective singularities (a characteristic $p > 0$ analog of Du Bois singularities).

This naturally leads one to ask whether Fujita's conjecture \cite{FujitaOnPolarizedManifolds} (or the bounds in \cite{AngerhnSiuEffectiveFreeness}) might hold for $\FCohomology^0(X, \omega_X \tensor \sL)$.
\begin{question}
 Suppose that $X$ is a $d$-dimensional smooth variety and $\sL$ is any ample line bundle.  Is it true that for all $n \geq d$ that the linear system associated to:
\begin{itemize}
 \item[(i)]  $\FCohomology^0(X, \omega_X \tensor \sL^{n + 1})$ is base-point-free?
 \item[(ii)]  $\FCohomology^0(X, \omega_X \tensor \sL^{n + 2})$ induces an embedding?
\end{itemize}
\end{question}
\noindent Even without the assumption that $\sL$ is globally generated, we answer this question in the affirmative for curves in Theorem \ref{thm.TauCohomologyForCurves}.

This special vector subspace $\FCohomology^0(X, \omega_X \tensor \sL)$ also behaves better with respect to restriction than $H^0\left(X, \omega_X \tensor \sL\right)$.  Suppose that $Z$ is any $F$-pure center of a pair $(X, \Delta)$.  There is a surjection between a canonical vector subspace on $X$ and another one on $Z$, for any line bundle $L$ such that $L - K_X - \Delta$ is ample.  This, combined with the previous results implies:

\begin{theorem*}
[Corollary \ref{cor.NoBasePointsAlongZ}]
Suppose that $X$ is a normal projective variety and $\Delta$ is a $\bQ$-divisor such that $K_X + \Delta$ is $\bQ$-Cartier with index not divisible by $p > 0$ and that $(X, \Delta)$ is sharply $F$-pure.  Further suppose that $M$ is a Cartier divisor on $X$ such that $M - K_X - \Delta$ is ample.  Suppose that $Z \subseteq X$ is a normal $F$-pure center of $(X, \Delta)$.  Write $(K_X + \Delta)|_Z \sim_{\bQ} K_Z + \Delta_Z$ with $\Delta_Z$ as in Definition \ref{def.FPureCenter}.  Finally suppose that one of the following holds:
\begin{itemize}
\item[(i)]  $\O_Z(M)$ is equal to $\O_Z((\dim Z) A + B)$ for some globally generated ample Cartier divisor $A$ on $Z$ and some Cartier divisor $B$ such that $B - K_Z - \Delta_Z$ is ample, or
\item[(ii)]  $Z$ is a minimal $F$-pure center and $\O_Z(M)$ is equal to $\O_Z((\dim Z) A + B)$ for some globally generated ample Cartier divisor $A$ on $Z$ and some Cartier divisor $B$ such that $B - K_Z - \Delta_Z$ is big and nef, or
\item[(iii)]  $Z$ is a smooth curve and $M|_Z - K_Z - \Delta_Z \sim_{\bQ} (M - K_X - \Delta)|_Z$ has degree $>1$, or
\item[(iv)]  $Z$ is a closed point.
\end{itemize}
Then $H^0(X, \O_X(M))$ has no base points along $Z$.
\end{theorem*}

In the final section of this paper, Section \ref{sec.Comments}, we briefly point out that many of the results of this paper admit generalizations to the setting of Cartier modules as developed in \cite{BlickleBoeckleCartierModulesFiniteness}.  Indeed, the modules described in \cite{BlickleBoeckleCartierModulesFiniteness} seem to be the natural ones from the perspective of global generation statements.

\vskip 12pt
\noindent\parbox{6in}{\noindent {\it Acknowledgements: }
\vskip 3pt
\hskip 11pt
The author would like to thank Manuel Blickle, H{\'e}l{\`e}ne Esnault, Christopher Hacon, Mircea Musta{\c{t}}{\u{a}}, Karen Smith, Kevin Tucker and Wenliang Zhang for inspiring conversations.  The author would like to thank the referee, Christopher Hacon, Zsolt Patakfalvi and Kevin Tucker for numerous helpful comments on previous drafts of this paper.  He is also very thankful to Brian Harbourne for pointing out some additional references in characteristic $p > 0$ related to Theorem \ref{thm.HypersurfaceDegreeApplication}.  He is also thankful to Shunsuke Takagi and Nobuo Hara for some discussion of the history of the simple proof of special cases of Fujita's conjecture in characteristic $p > 0$.   Large parts of this paper were written while visiting the Johannes Gutenberg-Universit\"at Mainz during the summer of 2011.  This visit was partially funded by the SFB/TRR45 \emph{Periods, moduli, and the arithmetic of algebraic varieties}.}

\section{Preliminaries}
\label{sec.Preliminaries}

Throughout this paper, all schemes will be assumed to be separated and of finite type over an $F$-finite field $k$ of characteristic $p > 0$.  Note every perfect field is automatically $F$-finite.  A \emph{variety} will mean a connected reduced equidimensional scheme over $k$ (note that we do not assume varieties are irreducible).  For any variety of dimension $d$, we use $\omega_X$ to denote $\myH^{-d} \omega_X^{\mydot}$ (here $\omega_X^{\mydot}$ is used to denote $\eta^! k$ where $\eta : X \to k$ is the structural map).  We point out also that little is lost by simply restricting to normal irreducible varieties.  If $X$ is normal, then a \emph{canonical divisor} on $X$ is any divisor $K_X$ such that $\omega_X \cong \O_X(K_X)$.

We also record the following special case of Castelnuovo-Mumford regularity which we will frequently use.

\begin{theorem}[Mumford] \cf \cite[Theorem 1.8.5]{LazarsfeldPositivity2}
\label{thm.Regularity}
 Suppose that $X$ is a projective variety and $\sM$ is a globally generated ample line bundle on $X$.  Additionally suppose that $\sF$ is a coherent sheaf on $X$ such that $H^i(X, \sF \tensor \sM^{-i}) = 0$ for all $i > 0$.  Then $\sF$ is globally generated.
\end{theorem}

The rest of this section is somewhat technical and deals with $\bQ$-divisors and test ideals.  Thus we suggest that the reader may wish to read Section \ref{MotvationalComputations} below first in order to learn the motivating ideas of this paper.

On a normal variety, a $\bQ$-divisor $\Delta$ is a divisor with rational coefficients, $\Delta = \sum a_i D_i$ where the $D_i$ are prime divisors.  With such a $\Delta$, we use $\lceil \Delta \rceil$ to denote $\sum \lceil a_i \rceil D_i$.  A $\bQ$-divisor $\Gamma$ is called \emph{$\bQ$-Cartier} if there exists an integer $n > 0$ such that $n\Gamma$ has integer coefficients and as such is a Cartier divisor.  In this case, the smallest such $n > 0$ is called the \emph{index of $\Gamma$}.

\begin{definition}
A \emph{pair $(X, \Delta)$} is the combined information of a normal integral scheme $X$ and an effective $\bQ$-divisor $\Delta$.
\end{definition}

Our next order of business is to recall the definition of the test ideal and parameter test ideal of a pair $(X, \Delta)$, as well as other notions of $F$-singularities.
Consider now the $e$-iterated absolute Frobenius map $F^e : X \to X$.  The source of $F^e$ is isomorphic to $X$, although as a variety over $k$, $\eta \circ F^e : X \to k$, it is different than $X$.  We notice that $(F^e)^! \omega_X^{\mydot} = (F^e)^! \eta^! k = \eta^! (F^e)^! k \cong \eta^! k = \omega_X^{\mydot}$.  Thus we have a trace map $F^e_* \omega_X^{\mydot} \to \omega_X^{\mydot}$. Taking cohomology yields $\Phi^e : F^e_* \omega_X \to \omega_X$.

\begin{definition}[\cite{SmithTestIdeals, BlickleMultiplierIdealsAndModulesOnToric, SchwedeTakagiRationalPairs}] \label{def.ParTestSub}
\label{def.ParamTestSubmodule}
Suppose that $X$ is a variety.  The \emph{parameter test submodule of $X$}, denoted $\tau(\omega_X)$, is the unique smallest $\O_X$-submodule $M$ of $\omega_X$, non-zero on any component of $X$, such that $\Phi^1(F_* M) \subseteq M$ where $\Phi^1 : F_* \omega_X \to \omega_X$ is the canonical dual of Frobenius.

Let $X = \Spec R$ be a normal variety and let $\Gamma \geq 0$ is a $\bQ$-divisor on $X$.  The \emph{parameter test submodule of $(X, \Gamma)$}, denoted $\tau(\omega_X, \Gamma)$, is the unique smallest non-zero $\O_X$-submodule of $\omega_X$ such that $\phi(F^e_* M) \subseteq M$ where $\phi$ ranges over all $\phi \in \Hom_{\O_X}(F^e_* \omega_X( \lceil (p^e - 1)\Gamma \rceil), \omega_X)$ and all $e > 0$ (notice $\omega_X \subseteq \omega_X( \lceil (p^e - 1)\Gamma \rceil)$).

If $\Gamma$ is not necessarily effective, then choose $A$ a Cartier divisor such that $\Gamma + A$ is effective (working locally if necessary).  We then define $\tau(\omega_X, \Gamma) = \tau(\omega_X, \Gamma + A) \tensor \O_X(A)$.  This is independent of the choice of $A$.
\end{definition}

Both $\tau(\omega_X)$ and $\tau(\omega_X, \Delta)$ exist and their formation commutes with localization and so they can be defined on any (not necessarily affine) variety.

\begin{definition}
 Suppose that $X$ is a normal variety and $\Delta$ is a $\bQ$-divisor on $X$.  Fix a canonical divisor $K_X$.  We define the test ideal $\tau(X, \Delta)$ to be $\tau(\omega_X, K_X + \Delta)$.
\end{definition}
Note $\tau(X, \Delta)$ is always an ideal sheaf of $X$ if $\Delta \geq 0$, see \cite[Lemma 2.29]{BlickleSchwedeTuckerTestAlterations} (otherwise, it is a fractional ideal sheaf).  We also note that $\tau(X, \Delta)$ is actually what is typically called the \emph{big test ideal} or \emph{non-finitistic test ideal} and denoted by $\tau_b(X, \Delta)$ or $\tld{\tau}(X, \Delta)$.  However, in the case that $K_X + \Delta$ is $\bQ$-Cartier, the only setting considered in this paper, $\tau_b(X, \Delta)$ and the classically defined $\tau(X, \Delta)$ coincide, \cite{TakagiInterpretationOfMultiplierIdeals, BlickleSchwedeTakagiZhang}.

We also record the following useful facts about (parameter) test ideals:
\begin{proposition} \cf \cite{BlickleMustataSmithDiscretenessAndRationalityOfFThresholds,BlickleSchwedeTakagiZhang, SchwedeTakagiRationalPairs,SchwedeTuckerTestIdealSurvey}
 Suppose that $(X, \Delta)$ is as above and fix $A$ to be a Cartier divisor on $X$.  Then
\begin{itemize}
 \item[(a)]  $\tau(X, \Delta + A) = \tau(X, \Delta) \tensor \O_X(-A)$.
 \item[(b)]  For all $1 \gg \varepsilon > 0$, we have that $\tau(X, \Delta + \varepsilon A) = \tau(X, \Delta)$.
 \item[(c)]  $\tau(\omega_X, \Delta + A) = \tau(\omega_X, \Delta) \tensor \O_X(-A)$.
 \item[(d)]  For all $1 \gg \varepsilon > 0$, we have that $\tau(\omega_X, \Delta + \varepsilon A) = \tau(\omega_X, \Delta)$.
\end{itemize}
\end{proposition}

Suppose now that $(X, \Delta \geq 0)$ is a pair such that $K_X + \Delta$ is $\bQ$-Cartier with index not divisible by $p > 0$. Choose $e > 0$ such that $(p^e - 1)(K_X + \Delta)$ is Cartier.  Then associated to $\Delta$ and $e$ we obtain a canonically determined (up to multiplication by a unit) map $\phi_{\Delta} : F^e_* \sL_{e, \Delta} \to \O_X$, where $\sL_{e, \Delta} = \O_X( (1-p^e)(K_X + \Delta))$ is a line bundle.  Conversely, given a map $\phi : F^e_* \sL \to \O_X$, it uniquely determines an effective $\bQ$-divisor $\Delta_{\phi}$ such that $\O_X\left( (1-p^e)(K_X + \Delta_{\phi}) \right) \cong \sL$. See \cite{SchwedeFAdjunction} or \cite{SchwedeTuckerTestIdealSurvey} for further discussion.  In this context however, $\tau(X, \Delta)$ is the unique smallest non-zero ideal $J \subseteq \O_X$ such that
\[
\phi_{\Delta} \big(F^e_* (J \cdot \sL_{e, \Delta} )\big) \subseteq J.
\]
In fact, it is easy yet important to observe that we have equality:
\begin{equation}
 \phi_{\Delta} \big(F^e_* (J \cdot \sL_{e, \Delta} )\big) = J.
\end{equation}

We also record the following fact for future use.  First recall that if $x \in X$ is a smooth point and $D$ is a $\bQ$-divisor on $X$, then we define the multiplicity of $D$ at $x$, denoted $\mult_x D$ to be ${1 \over n} \mult_x (nD)$ where $n > 0$ is such that $nD$ is Cartier at $x \in X$.  See for example \cite[Page 163]{LazarsfeldPositivity2}.

\begin{lemma} \cf \cite[Proposition 9.3.2]{LazarsfeldPositivity2}
\label{lem.MultiplicityImpliesTestContainment}
Fix $X$ to be a smooth $n$-dimensional variety and suppose that $D$ is a $\bQ$-divisor with multiplicity $\geq l$ at a (possibly non-closed) point $x \in X$ of codimension $l$.  Then $\tau(X, D) \subseteq \bq_x$ where $\bq_x$ is the prime ideal sheaf corresponding to $x \in X$.
\end{lemma}
\begin{proof}
Suppose $\pi : Y \to X$ is the blow up of the point $x \in X$ and set $E$ to be the component of the exceptional divisor dominating $x$.  Observe that $E$'s coefficient in $K_Y - \pi^* (K_X + D)$ is $\leq l - 1 - l = -1$.  Therefore, $\pi_* \O_Y(\lceil K_Y - \pi^*(K_X + D) \rceil) \subseteq \bq_x$.  It follows from the usual arguments relating test ideals and multiplier ideals that $\tau(X, D) \subseteq \pi_* \O_Y(\lceil K_Y - \pi^*(K_X + D) \rceil)$, \cf \cite{SchwedeTuckerTestIdealSurvey, HaraWatanabeFRegFPure}, and the result follows.
\end{proof}

\begin{definition} \cite{FujinoSchwedeTakagiSupplements, BlickleBoeckleCartierModulesFiniteness, BlickleTestIdealsViaAlgebras}
 Suppose that $(X, \Delta \geq 0)$ is a pair such that $K_X + \Delta$ is $\bQ$-Cartier with index not divisible by $p$.  Fix $\sL = \sL_{e, \Delta}$ and $\phi_\Delta : F^e_* \sL \to \O_X$ as above.  Then we define the \emph{non-$F$-pure ideal of $(X, \Delta)$}, denoted $\sigma(X, \Delta)$ or $\sigma(X, \phi_{\Delta})$, to be the unique largest ideal $J$ of $\O_X$ such that
\[
\phi_{\Delta} \big(F^e_* (J \cdot \sL )\big) = J.
\]
For any (possibly non-normal) reduced scheme over $k$ with a given $\phi : F^e_* \sL \to \O_X$, we define $\sigma(\phi)$ similarly.

For an arbitrary variety $X$, we define $\sigma(\omega_X)$ to be the unique largest submodule $M$ of $\omega_X$ such that $\Phi (F_* M) = M$ where $\Phi : F^e_* \omega_X \to \omega_X$ is as above.
\end{definition}

\begin{remark}
While we defined $\tau(X, \Delta)$ for possible non-effective $\Delta$, we only define $\sigma(X, \Delta)$ for $\Delta \geq 0$.
\end{remark}

\begin{remark}
Alternately, set $J_0 = \O_X$ and define recursively $J_{n + 1} := \phi_{\Delta}(F^e_* (J_n \cdot \sL))$.  Then it can be shown that $J_n = J_{n + 1} = \ldots$ for $n \gg 0$ \cite{Gabber.tStruc}, \cf \cite{BlickleSchwedeTakagiZhang,LyubeznikFModulesApplicationsToLocalCohomology,HartshorneSpeiserLocalCohomologyInCharacteristicP}.  It follows immediately that $\sigma(X, \Delta) = J_n$ for $n \gg 0$.  Likewise, $\sigma(\omega_X) = \Phi^e(F^e_* \omega_X)$ for $e \gg 0$.
\end{remark}

\begin{definition}
With notation as above:
\begin{itemize}
 \item[(i)]  we say that $(X, \Delta \geq 0)$ is \emph{strongly $F$-regular}, if $\tau(X, \Delta) = \O_X$
 \item[(ii)]  we say that $X$ is \emph{$F$-rational} if $\tau(\omega_X) = \omega_X$.
 \item[(iii)]  we say that $(X, \Delta \geq 0)$ is \emph{sharply $F$-pure} if $\sigma(X, \Delta) = \O_X$.
 \item[(iv)]  we say that $X$ is \emph{$F$-injective} if $\Phi^i : \myH^i(F^e_* \omega_X^{\mydot}) \to \myH^i(\omega_X^{\mydot})$ is surjective for each $i$.
\end{itemize}
\end{definition}

\begin{remark}
Strongly $F$-regular pairs are the analog of log terminal pairs in characteristic zero \cite{HaraWatanabeFRegFPure}.  $F$-rational varieties are the analog of varieties with rational singularities in characteristic zero \cite{SmithFRatImpliesRat, HaraRatImpliesFRat, MehtaSrinivasRatImpliesFRat}.  $F$-pure pairs are the analog of log canonical pairs in characteristic zero \cite{HaraWatanabeFRegFPure}.  Finally, $F$-injective varieties are the analog of varieties with Du Bois singularities in characteristic zero \cite{SchwedeFInjectiveAreDuBois}.
\end{remark}

\section{Motivational computations and easy cases}
\label{MotvationalComputations}

Suppose that $X$ is a smooth (or simply normal) projective algebraic variety over an algebraically closed field $k$ of characteristic $p > 0$.  Further suppose that $\sL$ is any line bundle on $X$.  Consider the two vector subspaces:
\[
{\small
\begin{array}{rccccl}
 \tauCohomology^0(X, \omega_X \tensor \sL) & = & \bigcap_{f : Y \to X} & \Image\Big(H^0(X, f_* (\omega_Y \tensor_{\O_Y} f^* \sL) ) & \to & H^0(X, \omega_X \tensor_{\O_X} \sL) \Big)\\
 \FCohomology^0(X, \omega_X \tensor \sL) & = & \bigcap_{e \geq 0} & \Image \Big( H^0(X, F^e_* (\omega_X \tensor_{\O_X} \sL^{p^e}) ) & \to & H^0(X, \omega_X \tensor_{\O_X} \sL) \Big)
 \end{array}
}
\]
where the first intersection runs over all finite maps $f : Y \to X$ of irreducible varieties (note it is harmless to only consider reduced or even normal $Y$).  This first vector subspace $T^0$ was explored in \cite{BlickleSchwedeTuckerTestAlterations} where it was used to construct transformation rules for test ideals under birational maps.

Notice that we always have
\[
\tauCohomology^0(X, \omega_X \tensor \sL) \subseteq \FCohomology^0(X, \omega_X \tensor \sL)
\]
Furthermore, because $H^0(X, \omega_X \tensor_{\O_X} \sL)$ is a finite dimensional $k$-vector space, we see that both intersections actually stabilize.  In other words, both intersections are equal to one of their members.

We are primarily interested in the case where $\sL$ is ample or in some weaker sense positive.  Thus as a first step, consider the following:

\begin{lemma}
\label{lem.StabilizingTauForBigN}
Suppose that $X$ is a smooth projective variety and $\sL$ is an ample line bundle.  Then there exists an integer $n_0$, depending on $X$ and $\sL$, such that
\[
\tauCohomology^0(X, \omega_X \tensor \sL^n) = \FCohomology^0(X, \omega_X \tensor \sL^n) = H^0(X, \omega_X \tensor_{\O_X} \sL^n)
\]
for all $n \geq n_0$.
In other words, the two linear systems associated to $\tauCohomology^0(X, \omega_X \tensor \sL^n)$ and $\FCohomology^0(X, \omega_X \tensor \sL^n)$ are complete linear systems.
\end{lemma}
\begin{proof}
Consider the section ring $S = \oplus_{n \in \bZ} H^0(X, \sL^n)$.  The canonical module of this ring is $\omega_S = \oplus_{n \in \bZ} H^0(X, \omega_X \tensor \sL^n)$.

We first show that $\FCohomology^0(X, \omega_X \tensor \sL^n) = H^0(X, \omega_X \tensor_{\O_X} \sL^n)$ since that is conceptually easier.
Now the maps $\Phi^e : F^e_* \omega_S \to \omega_S$ are induced by the maps $\phi^e : H^0(X, F^e_* (\omega_X \tensor_{\O_X} \sL^{mp^e}) ) \to H^0(X, \omega_X \tensor_{\O_X} \sL^m)$ in the obvious way, sending $H^0(X, F^e_* (\omega_X \tensor_{\O_X} \sL^{mp^e}) ) \to H^0(X, F^e_* (\omega_X \tensor_{\O_X} \sL^{m}) )$ and sending $H^0(X, F^e_* (\omega_X \tensor_{\O_X} \sL^{n}) )$ to zero for $n$ not divisible by $p^e$, \cf \cite[Section 5.3]{SchwedeSmithLogFanoVsGloballyFRegular}.

Now, the image of $\Phi^e$ stabilizes for $e \gg 0$ by \cite{Gabber.tStruc}, \cf \cite{BlickleSchwedeTakagiZhang}.  In other words, $\Phi^e(F^e_* \omega_S) = \Phi^{e+1}(F^{e+1}_* \omega_S)$ for all sufficiently large $e$.  Furthermore, it is surjective outside of the irrelevant ideal $S_+ = \oplus_{n \geq 1} S_n$ because $S$ is smooth on the punctured spectrum.  Thus $\bigcap_{e \geq 0} \Phi^e(F^e_* \omega_S)$ is contained within $S_{\geq n_0} \cdot \omega_S$ for some $n_0 \geq 0$.  But this implies that $\FCohomology^0(X, \omega_X \tensor \sL^n) = H^0(X, \omega_X \tensor_{\O_X} \sL^n)$ for all $n > n_0$.

For $\tauCohomology^0(X, \omega_X \tensor \sL^n)$ we use a similar idea.   However, in this case it follows that
\[
\oplus_{n \in \bZ} \tauCohomology^0(X, \omega_X \tensor \sL^n) \supseteq \tau(\omega_S)
\]
where $\tau(\omega_S)$ is the parameter-test submodule \cite{BlickleSchwedeTuckerTestAlterations}.  Again, because $S$ is smooth on the punctured spectrum, $\tau(\omega_S) \supseteq S_{\geq m_0} \cdot \omega_S$ for some $m_0 > 0$ and so $\tauCohomology^0(X, \omega_X \tensor \sL^n) = H^0(X, \omega_X \tensor_{\O_X} \sL^n)$ for all $n > m_0$.
\end{proof}

\begin{remark}
In fact, the above proof shows that the condition that $X$ is smooth can be weakened substantially.  In particular, one still has $\FCohomology^0(X, \omega_X \tensor \sL^n) = H^0(X, \omega_X \tensor_{\O_X} \sL^n)$ for $n \gg 0$ if one assumes that $X$ has $F$-injective singularities.  More generally one has $\tauCohomology^0(X, \omega_X \tensor \sL^n) = H^0(X, \omega_X \tensor_{\O_X} \sL^n)$ for $n \gg 0$ if one assumes that $X$ has $F$-rational singularities.
\end{remark}

\subsection{Curves}

Suppose that $C$ is a smooth curve of genus $g$ and further suppose that $\sL$ is an ample line bundle of degree $d > 0$ on $C$.  The map
\[
\phi_{C, \sL} : H^0(C, F^e_* (\omega_C \tensor \sL^{p^e}) ) \to H^0(C, \omega_C \tensor \sL)
\]
need not be surjective in general, even if $H^0(C, F^e_* \omega_C) \to H^0(C, \omega_C)$ is surjective, see \cite{TangoOnTheBehaviorOfVBAndFrob}.  However, if $d \geq {2g - 2 \over p}$, then it is surjective by \cite[Lemma 10]{TangoOnTheBehaviorOfVBAndFrob}.  Thus $\FCohomology^0(X, \omega_X \tensor \sL) = H^0(X, \omega_X \tensor_{\O_X} \sL)$ for all $\sL$ of degree at least ${2g - 2 \over p}$.  However, even if it is not surjective one can still ask whether or not the linear system associated to $\FCohomology^0(X, \omega_X \tensor \sL)$ is base-point-free or more generally induces an embedding.  

\begin{theorem}
\label{thm.TauCohomologyForCurves}
With notation as above, if $\sL$ has degree at least 2, then the linear system associated to $\tauCohomology^0(C, \omega_C \tensor \sL) \subseteq \FCohomology^0(C, \omega_C \tensor \sL)$ is base-point-free.  If $\sL$ has degree at least 3, then $\tauCohomology^0(C, \omega_C \tensor \sL) \subseteq \FCohomology^0(C, \omega_C \tensor \sL)$ induces an embedding of $C$ into projective space.
\end{theorem}
Notice that $\deg (\omega_C \tensor \sL) = 2g - 2 + d$.  Also recall that for any line bundle $N$ of degree $2g - 2 + d$, the linear system $|\omega_X \tensor_{\O_X} \sL |$ is base-point-free for $d \geq 2$ and induces an embedding for $d \geq 3$ \cite[Chapter IV, Section 3]{Hartshorne}.  Of course, one might hope for better bounds \cite[Chapter IV, Section 5]{Hartshorne}.
\begin{proof}
First suppose $\deg \sL = 2$.  Further, choose any point $Q \in C$.  It is sufficient to show that the map $\alpha$ in the diagram below:
\[
{\small
\xymatrix@C=15pt{
0 \ar[r] & H^0(C, \omega_C(-Q) \tensor \sL) \ar[r] & H^0(C, \omega_C \tensor \sL) \ar[r]^{\alpha} & H^0(Q, \omega_Q \tensor \sL) \ar[r] & 0 = H^1(C, \omega_C(-Q) \tensor \sL)
}
}
\]
does not send $\tauCohomology^0(C, \omega_C \tensor \sL)$ to zero.  In other words, this means that $\alpha|_{\tauCohomology^0(C, \omega_C \tensor \sL)}$ is surjective.  Here the final entry $H^1(C, \omega_C(-Q) \tensor \sL)$ is zero because $\deg \sL(-Q) > 0$.

Fix $f : D \to C$ to be any finite cover of $C$ from a normal scheme $D$ such that the image of $H^0(D, \omega_D \tensor_{\O_D} f^*\sL) \to H^0(C, \omega_C \tensor \sL)$ is equal to $\tauCohomology^0(C, \omega_C \tensor \sL)$.  Consider the following diagram:
\[
{\small
\xymatrix@C=15pt{
0 \ar[r] & H^0(C, \omega_C(-Q) \tensor \sL) \ar[r] & H^0(C, \omega_C \tensor \sL) \ar[r]^{\alpha} & H^0(Q, \omega_Q \tensor \sL) \ar[r] & 0  \\
0 \ar[r] & H^0(D, \omega_D(-f^*Q) \tensor_{\O_D} f^*\sL) \ar[u] \ar[r] & H^0(D, \omega_D \tensor_{\O_D} f^*\sL) \ar[u]^{\phi} \ar[r]^-{\beta} & H^0(f^* Q, \omega_{f^* Q} \tensor_{\O_D} f^* \sL) \ar[u]^{\psi} \ar[r] & 0
}
}
\]
noticing the $f^* Q$ need not be reduced or irreducible.  Here the top row is exact since $0 = H^1(C, \omega_C(-Q) \tensor \sL)$ and the bottom row is exact since $0 = H^1(D, \omega_D(-f^*Q) \tensor_{\O_D} f^*\sL)$.  However, $Q$ is a point, and thus non-singular and so the natural map $f_* \omega_{f^* Q} \to \omega_Q$ is surjective.  But this implies $\psi$ is surjective and so $\psi \circ \beta = \alpha \circ \phi$ is surjective as well.  It immediately follows that $\alpha|_{\Image(\phi)}$ is still surjective.  This completes the proof of the first statement.

For the second part, choose two points $P, Q \in C$ (possibly allowing $P = Q$).  We form a diagram as before:
\[
{\scriptsize
\xymatrix@C=10pt{
0 \ar[r] & H^0(C, \omega_C(-P - Q) \tensor \sL) \ar[r] & H^0(C, \omega_C \tensor \sL) \ar[r]^{\alpha} & H^0(C, \omega_{P + Q} \tensor \sL) \ar[r] & 0  \\
0 \ar[r] & H^0(D, \omega_D(-f^*(P +Q) ) \tensor_{\O_D} f^*\sL) \ar[u] \ar[r] & H^0(D, \omega_D \tensor_{\O_D} f^*\sL) \ar[u]^{\phi} \ar[r]^-{\beta} & H^0(f^* (P+Q), \omega_{f^* (P+Q)} \tensor_{\O_D} f^* \sL) \ar[u]^{\psi} \ar[r] & 0
}
}
\]
As before, the top row is exact since $0 = H^1(C, \omega_C(-P-Q) \tensor \sL)$ and the bottom row is exact since $0 = H^1(D, \omega_D(-f^*(P+Q)) \tensor_{\O_D} f^*\sL)$.
Again, we want to prove that $\alpha|_{\Image(\phi)}$ is surjective.
If $P \neq Q$, then the same argument as before implies the desired surjectivity.  Now suppose that $P = Q$.  It is sufficient to prove that $f_* \omega_{f^* 2Q} \to \omega_{2Q}$ is surjective.  Of course, there are covers of $\omega_{2Q}$ where that map is not surjective (for example, the one corresponding to the reduced scheme $Q$).  Consider the map $\O_{C, Q} \to \O_{D, (f^{-1}Q)_{\reduced}}$ of semi-local rings.  Since $\O_{C, Q}$ is regular, this map is split.  It follows that $\O_{2Q} \to \O_{2f^*Q}$ is also split.  The desired surjectivity now immediately follows.
\end{proof}


\subsection{Global generation in higher dimensions}

Suppose that $\sL$ is a globally generated ample line bundle on a smooth (or even $F$-injective) variety $X$ of dimension $d$.  In \cite{SmithFujitasFreeness,HaraACharacteristicPAnalogOfMultiplierIdealsAndApplications, KeelerFujita}, it was shown that $|\omega_X \tensor \sL^{d + 1}|$ is base-point-free, respectively that $|\omega_X \tensor \sL^{d + 2}|$ induces an embedding of $X$ into projective space.  Further refinements are also possible (replacing $\sL^{d + 1}$ by $\sL^{d} \tensor \sM$ for some ample line bundle $\sM$ \cite{KeelerFujita}).  However, those previous proofs actually showed something more, they in fact showed that the linear system associated to $\FCohomology^0(X, \omega_X \tensor \sL^{d+1})$ was base-point-free and that $\FCohomology^0(X, \omega_X \tensor \sL^{d+2})$ induced an embedding into projective space.

In order to motivate our later work, let us briefly prove this result since the proof is quite straightforward.

\begin{theorem}  \cite{SmithFujitasFreeness,HaraACharacteristicPAnalogOfMultiplierIdealsAndApplications, KeelerFujita}
\label{thm.KeelerThm}
Suppose that $X$ is a projective $F$-injective variety and that $\sL$ is any globally generated ample line-bundle and $\sM$ is any other ample line bundle.  Then
\begin{itemize}
\item[(1)]  $\omega_X \tensor \sL^{d} \tensor \sM$ is globally generated by $\FCohomology^0(X, \omega_X \tensor \sL^{d} \tensor \sM)$.
\item[(2)]  If $X$ is smooth, then $\FCohomology^0(X, \omega_X \tensor \sL^{d+1} \tensor \sM)$ induces an embedding of $X$ into projective space.
\end{itemize}
\end{theorem}
\begin{proof} [Proof taken from \cite{KeelerFujita}]
Choose $e \gg 0$ and notice that $H^i(X, \omega_X \tensor \sM^{p^e} \tensor \sL^{p^e (d - i)}  ) = 0$ for all $i > 0$ by Serre vanishing.  Therefore, $(F^e_* \omega_X) \tensor \sM \tensor \sL^{d}$ is globally generated \emph{as an $\O_X$-module} by Theorem \ref{thm.Regularity}.  However, $(F^e_* \omega_X) \tensor \sM \tensor \sL^d$ surjects onto $\omega_X \tensor \sM \tensor \sL^d$ by the assumption on the singularities of $X$.  In particular, $\omega_X \tensor \sM \tensor \sL^d$ is globally generated by the linear system associated to
\[
\Image \left( H^0(X, (F^e_* \omega_X)\tensor \sM \tensor \sL^d) \to H^0(X, \omega_X \tensor \sM \tensor \sL^d) \right)
\]
This subspace equals $\FCohomology^0(X, \omega_X \tensor \sL^{d} \tensor \sM)$ since $e \gg 0$.

For the second half, mimicking the proof of \cite{KeelerFujita}, \cf \cite[Example 1.8.22]{LazarsfeldPositivity2}, we fix a point $y \in X$ and choose $e \gg 0$ such that
\[
 \bm_y \tensor (F^e_* \omega_X) \tensor \sL^{d+ 1} \tensor \sM
\]
is globally generated (this was shown in \cite{KeelerFujita}).  Consider the following diagram:
\[
 \xymatrix{
H^0(X, \bm_y \tensor (F^e_* \omega_X) \tensor \sL^{d+ 1} \tensor \sM) \ar[d]_{\gamma} \ar[r]^-{\alpha} & H^0(X, (F^e_* \omega_X) \tensor \sL^{d+ 1} \tensor \sM) \ar[d]^{\delta} \\
H^0(X, \bm_y \tensor \omega_X \tensor \sL^{d+ 1} \tensor \sM) \ar[r]_-{\beta} & H^0(X, \omega_X \tensor \sL^{d+ 1} \tensor \sM)
}
\]
It immediately follows that $\beta^{-1}(S^0\big(X, \omega_X \tensor \sL^{d+1} \tensor \sL)\big)$ globally generates $\bm_y \tensor \omega_X \tensor \sL^{d+ 1} \tensor \sM$.  Note this holds for each point $y \in X$ so that $S^0\big(X, \omega_X \tensor \sL^{d+1} \tensor \sL)$ clearly separates points.  It also separates tangent vectors since $\beta^{-1}(S^0\big(X, \omega_X \tensor \sL^{d+1} \tensor \sL)\big)$ globally generates $\bm_y \tensor \omega_X \tensor \sL^{d+ 1} \tensor \sM$ and thus its image also globally generates $(\bm_y / \bm_y^2) \tensor \omega_X \tensor \sL^{d+ 1} \tensor \sM$, \cf \cite[Page 21]{AmpleSubvarietiesOfAlgebraicVars}.
\end{proof}

\section{$\bQ$-divisors and global generation}

\begin{definition}
\label{def.QDivisorLinearSystem}
Suppose that $X$ is a normal proper variety over an algebraically closed field of characteristic $p > 0$.  Further suppose $\Delta \geq 0$ is a $\bQ$-divisor such that $K_X + \Delta$ is $\bQ$-Cartier with index not divisible by $p$.  Finally suppose that $M$ is any Cartier divisor.

Consider the map $\phi_{\Delta}^e : F^e_* \sL_{e, \Delta} \to \O_X$ as in Section \ref{sec.Preliminaries}.  Notice that $\phi_{\Delta}$ restricts to surjective maps:
    \[\begin{split}
        F^e_* \left( \sigma(X, \Delta) \tensor \sL_{e, \Delta} \right) \to \sigma(X, \Delta) \\
        F^e_* \left( \tau(X, \Delta) \tensor \sL_{e, \Delta} \right) \to \tau(X, \Delta)
    \end{split}
    \]
We then define $\FCohomology^0(X, \sigma(X, \Delta) \tensor \O_X(M) )$ as
\[
\begin{array}{rcl}
:= & \bigcap_{n \geq 0} & \Image \Big( H^0\big(X,  F^{ne}_* \sigma(X, \Delta) \tensor \sL_{ne, \Delta}(p^{ne} M)\big) \\
& & \to  H^0\big(X, \sigma(X, \Delta) \tensor \O_X(M)\big) \Big)\\
= & \bigcap_{n \geq 0} & \Image \Big( H^0\big(X,  F^{ne}_* \sL_{ne, \Delta}(p^{ne} M)\big) \\
& & \to H^0\big(X, \O_X(M)\big) \Big)\\
\subseteq & H^0(X, \O_X(M)).
\end{array}
\]
Likewise, we define $\FCohomology^0(X, \tau(X, \Delta) \tensor \O_X(M) )$ as
\[
\begin{array}{rcl}
:= & \bigcap_{n \geq 0} & \Image \Big( H^0\big(X,  F^{ne}_* \tau(X, \Delta) \tensor \sL_{ne, \Delta}(p^{ne} M)\big)\\
& &  \to  H^0\big(X, \tau(X, \Delta) \tensor \O_X(M)\big) \Big)\\
\subseteq & H^0(X, \O_X(M)).
\end{array}
\]
It is straightforward to see that these objects are independent of the choice of $e$.
\end{definition}

Observe the following:

\begin{lemma}
\label{lem.BasicPropertiesOfS0}
With notation as in Definition \ref{def.QDivisorLinearSystem}, it follows that
\[
\FCohomology^0\left(X, \tau(X, \Delta) \tensor \O_X(M) \right) \subseteq \FCohomology^0\left(X, \sigma(X, \Delta) \tensor \O_X(M) \right)
\]
Furthermore, if $\Delta_1 \geq \Delta_2$ both satisfy $K_X + \Delta_i$ being $\bQ$-Cartier with index not divisible by $p$, then
\[
\FCohomology^0\left(X, \tau(X, \Delta_1) \tensor \O_X(M) \right) \subseteq \FCohomology^0\left(X, \tau(X, \Delta_2) \tensor \O_X(M) \right).
\]
The same statement holds for $\sigma(X, \Delta_i)$.
\end{lemma}
\begin{proof}
The first statement is obvious since $\tau(X, \Delta) \subseteq \sigma(X, \Delta)$.  For the second statement, choose a large $e = e_1 = e_2$ and then notice that $\phi_{\Delta_1}^e$ factors through $\phi_{\Delta_2}^e$.  \end{proof}

\subsection{General global generation statements for $\FCohomology^0$}
\label{sec.GeneralGlobalGenerationForFCohomology}

We now state our first global generation statement, the proof strategy is the same as in Theorem \ref{thm.KeelerThm}, \cf \cite{KeelerFujita}.

\begin{theorem}
\label{thm.BaseGlobalGenerationForSigmaTau}
Suppose that $X$ is a $d$-dimensional variety and that $\Delta$ is a $\bQ$-divisor such that $\Gamma = K_X + \Delta$ is $\bQ$-Cartier with index not divisible by $p > 0$.  Further suppose that $L$ is a Cartier divisor such that $L - K_X - \Delta$ is ample.  Finally suppose that $M$ is a globally generated ample divisor.  Then
\[
\sigma(X, \Delta) \tensor \O_X(L + n M)
\]
is globally generated for $n \geq d$ by $\FCohomology^0(X, \sigma(X, \Delta) \tensor \O_X(L + n M) )$.  Furthermore, the same result holds for $\tau(X, \Delta)$ in place of $\sigma(X, \Delta)$.

In particular, if $(X, \Delta)$ is sharply $F$-pure, then the linear system associated to
\[
H^0(X, \O_X(L + dM))
\]
is base-point-free.
\end{theorem}

\begin{proof}
Choose $e > 0$ such that $(p^e -1) (K_X + \Delta)$ is Cartier.

We start with the natural map $\phi_{\Delta} : F^{e}_* \sL_e \to \O_X$ where $\sL_e = \O_X((1 - p^e)(K_X + \Delta))$.   For each $n > 0$, we have the natural map $\phi_{\Delta}^n : F^{ne}_* \sL_{ne} \to \O_X$ where $\sL_{ne} = \O_X((1 - p^{ne})(K_X + \Delta))$.  We fix $n \gg 0$.  Notice that
\[
\phi_{\Delta}^n\big(F^{ne}_* (\sigma(X, \Delta) \cdot \O_X((1 - p^{ne})(K_X + \Delta))) \big) = \sigma(X, \Delta)
\]
Twisting by $\O_X(L + dM)$ gives us a surjective map
\[
F^{ne}_* \big( \sigma(X, \Delta) \cdot \O_X((1 - p^{ne})(K_X + \Delta) + p^{ne} L + p^{ne} dM) \big) \to \sigma(X, \Delta) \tensor  \O_X(L + dM).
\]
The left-side is equal to $F^{ne}_* \big( \sigma(X, \Delta) \tensor \O_X(L) \tensor \O_X((p^{ne} - 1)(L - K_X - \Delta) + p^{ne} dM) \big)$.
We will show that this sheaf is globally generated as an $\O_X$-module by using Theorem \ref{thm.Regularity}, Castelnuovo-Mumford regularity.
To see this, simply notice that
\[
\begin{array}{rl}
& H^i\Big(X, F^{ne}_* \big( \sigma(X, \Delta) \tensor \O_X(L) \tensor \O_X((p^{ne} - 1)(L - K_X - \Delta) + p^{ne} dM) \big) \tensor \O_X(-iM)\Big)\\
= & H^i\Big(X, F^{ne}_* \big( \sigma(X, \Delta) \tensor \O_X(L) \tensor \O_X((p^{ne} - 1)(L - K_X - \Delta) + p^{ne} (d - i)M) \big) \Big)\\
= & H^i\Big(X, F^{ne}_* \big( \sigma(X, \Delta) \tensor \O_X(L + (d-i)M) \tensor \O_X((p^{ne} - 1)(L - K_X - \Delta + (d - i)M)) \big) \Big)
\end{array}
\]
which vanishes by Serre vanishing for $n \gg 0$ since for any $i \leq d$, $L - K_X - \Delta + (d - i)M$ is ample.  It follows that
\[
F^{ne}_* \big( \sigma(X, \Delta) \tensor \O_X(L) \tensor \O_X((p^{ne} - 1)(L - K_X - \Delta) + p^{ne} dM) \big)
\]
is globally generated as an $\O_X$-module, as claimed.
Therefore, the quotient $\sigma(X, \Delta) \tensor \O_X(L + nM)$ is then also globally generated as an $\O_X$-module

For the statement involving $\tau$, replace $\sigma$ by $\tau$.  For the final statement, notice that $(X, \Delta)$ is sharply $F$-pure if and only if $\sigma(X, \Delta) = \O_X$.
\end{proof}

\begin{remark}
In fact, the same result also holds for the relative (adjoint-like) test ideals of \cite{TakagiPLTAdjoint,SchwedeFAdjunction} and more generally for many of the ideals considered in \cite{BlickleBoeckleCartierModulesFiniteness,BlickleTestIdealsViaAlgebras}.
\end{remark}

We state an easy but important corollary showing we can weaken the ampleness condition on $L - K_X - \Delta$ to big and nef, via a perturbation trick.

\begin{corollary}\cf \cite[Proposition 9.4.26]{LazarsfeldPositivity2}
\label{cor.FullGlobalGeneration}
Suppose that $X$ is a normal projective variety of dimension $d$ and that $M$ is a globally generated ample Cartier divisor.  Let $\Delta$ be an effective $\bQ$-divisor such that $K_X + \Delta$ is $\bQ$-Cartier and suppose that $L$ is a Cartier divisor such that $L - K_X - \Delta$ is big and nef.
Then
\[
\tau(X, \Delta) \tensor \O_X(L + nM)
\]
is globally generated for $n \geq d$.  In particular, if $(X, \Delta)$ is strongly $F$-regular, then $\O_X(L + nM)$ is globally generated.
\end{corollary}
\begin{proof}
Choose an effective Cartier divisor $D$ such that $D + K_X$ is effective.  Choose another effective Cartier divisor $E$ such that $L - K_X - \Delta - \varepsilon E$ is ample for all $1 \gg \varepsilon > 0$.  Now fix a rational $\varepsilon > 0$ such that $L - K_X - \Delta - \varepsilon E$ is ample and such that $\tau(X, \Delta + \varepsilon E) = \tau(X, \Delta)$.
Additionally suppose that $np^{e_0}(K_X + \Delta + \varepsilon E)$ is Cartier for some integer $n$ with $p\not{|} \,\,\, n$ and some integer $e_0 \geq 0$.  Choose $e \gg e_0$ and consider the effective $\bQ$-Cartier $\bQ$-divisor
\[
\Gamma := {D + K_X + \Delta + \varepsilon E \over (p^e - 1)}.
\]
Since $e \gg 0$, we know that $\tau(X, \Delta + \varepsilon E + \Gamma) = \tau(X, \Delta + \varepsilon E) = \tau(X, \Delta)$.  Additionally, $L - K_X - \Delta - \varepsilon E - \Gamma$ is also ample for the same reason.  On the other hand,
\[
\begin{array}{cl}
& n(p^e - 1)(K_X + \Delta + \varepsilon E + \Gamma) \\
 = & n(p^e - 1)(K_X + \Delta + \varepsilon E) + n(D + K_X + \Delta + \varepsilon E) \\
= & nD + np^e(K_X + \Delta + \varepsilon E)
\end{array}
\]
which is Cartier.  Thus the index of $K_X + \Delta + \varepsilon E + \Gamma$ is not divisible by $p$ and the result follows immediately from Theorem \ref{thm.BaseGlobalGenerationForSigmaTau}.
\end{proof}

\begin{remark}
We cannot use the same technique for $\sigma(X, \Delta)$ in place of $\tau(X, \Delta)$ in Corollary \ref{cor.FullGlobalGeneration}.  In particular, we cannot guarantee that $\sigma(X, \Delta + \varepsilon E) = \sigma(X, \Delta)$ no matter how small $\varepsilon > 0$ is.
\end{remark}

\begin{remark}
\label{rem.S0ForBigAndSemiample}
If in the previous corollary, we assume that $K_X + \Delta$ is $\bQ$-Cartier with index not divisible by $p > 0$, then we need not introduce the divisor $\Gamma$ at all.  Now observe that
\[
\FCohomology^0\left(X, \tau(X, \Delta + \varepsilon E) \tensor \O_X(L + nM)\right) \subseteq \FCohomology^0\left(X, \tau(X, \Delta) \tensor \O_X(L + nM)\right)
\]
and so it follows that $\FCohomology^0\left(X, \tau(X, \Delta) \tensor \O_X(L + nM)\right)$ globally generates $\tau(X, \Delta) \tensor \O_X(L + nM)$.  Likewise, even without the assumption on the index of $K_X + \Delta$, for any definition of $\FCohomology^0\left(X, \tau(X, \Delta) \tensor \O_X(L + nM)\right)$ satisfying Lemma \ref{lem.BasicPropertiesOfS0}, we obtain the global generation by the same argument.
\end{remark}

We now obtain a result which directly generalizes the global generation results of \cite{KeelerFujita} and \cite{SmithFujitasFreeness}.

\begin{corollary}
Suppose that $X$ is an $F$-rational projective $n$-dimensional variety, $M$ is a globally generated ample Cartier divisor and $N$ is any big and nef Cartier divisor.  Then $\omega_X(nM + N)$ is globally generated.
\end{corollary}
\begin{proof}
Fix a canonical divisor $K_X$.
Choose any effective Cartier divisor $A$ such that $A - K_X$ is effective.  Set $\Delta = A - K_X$ and set $L = A + N$.  Then
\[
L - K_X - \Delta = A + N - K_X - A + K_X = N
\]
is big and nef.  It follows that $\tau(X, \Delta) = \tau(\omega_X, K_X + \Delta) = \tau(\omega_X, A) = \tau(\omega_X) \tensor \O_X(-A) = \omega_X(-A)$ where the last equality comes because $X$ is $F$-rational. Thus $\tau(X, \Delta) \tensor \O_X(L + nM) = \omega_X(-A + L + nM) = \omega_X(N + nM)$ is globally generated.
\end{proof}

The following application is a demonstration of the utility of the above results.  We give a new proof of a characteristic $p > 0$ analog of a special case of the main result of \cite{EsnaultViehwegSurUneMinoration}, \cf \cite{WaldschmidtNombresTranscendants}.  Variants of this problem were also studied by Bombieri, Skoda, Demailly, W\"ustholz, Chudnovsky and others.  See \cite{EsnaultViehwegSurUneMinoration} and \cite[Section 10.1]{LazarsfeldPositivity2} for a discussion of the history of this problem in characteristic zero.  Recently, B.~Harbourne observed that the characteristic $p > 0$ result below could be easily obtained via \cite{HochsterHunekeComparisonOfSymbolic}, the simple argument is briefly described in a special case in \cite[Slide 17]{HarbourneAsymptoticInvariantsSlides} and in the introduction to \cite{HarbourneHunekeSymbolicPowersHighly}.

\begin{theorem} [\cite{HarbourneAsymptoticInvariantsSlides}, \cf \cite{EsnaultViehwegSurUneMinoration}]
\label{thm.HypersurfaceDegreeApplication}
Fix a reduced closed subscheme $S \subseteq \bP^n_k = X$, where $k$ is an algebraically closed field of characteristic $p > 0$, such that each irreducible component $Z \subseteq  S$ has codimension $\leq e$ in $\bP^n_k$.
Suppose that $A \subseteq \bP^n_k$ is a hypersurface of degree $d$ such that $\mult_x A \geq l$ for all $x \in S$.  Then $S$ lies on a hypersurface of degree $\lfloor {de \over l} \rfloor$.
\end{theorem}
\begin{proof}
This proof is taken from \cite[Proposition 10.1.1, Example 10.1.4]{LazarsfeldPositivity2}, \cf \cite{EsnaultViehwegSurUneMinoration, EsnaultViehwegLecturesOnVanishing}.
Set $D = {e \over l}A$.  Notice that $\mult_x D \geq e$ for all $x \in S$.
It follows from Lemma \ref{lem.MultiplicityImpliesTestContainment} that $\tau(X, D) \subseteq \mathcal{I}_S$ where $\mathcal{I}_S$ is the ideal defining $S$.  Set $\delta = \lfloor {de \over l} \rfloor$.


Set $H$ to be the hyperplane divisor on $X$ and consider $m H - K_X - D \sim_{\bQ} mH  +(n+1)H - {de \over l} H$ and observe it is ample as long as $m > {de \over l}   - (n+1)$.  Since $m$ is an integer, this happens if and only if $m \geq \lfloor {de \over l}  \rfloor + 1 - (n+1) = \delta - n$.  Choose $L = (\delta - n)H$ and notice that then
\[
\tau(X, D) \tensor \O_X(L + n H) = \tau(X, D) \tensor \O_X(\delta)
\]
is globally generated.  But $\tau(X, D) \subseteq \mathcal{I}_S$ and so there exists a non-zero section
\[
s \in H^0(X, \tau(X, D) \tensor \O_X(\delta)) \subseteq H^0(X, \mathcal{I}_S \tensor \O_X(\delta)) \subseteq H^0(X, \O_X(\delta)).
\]
The zero locus of that section has the desired property.
\end{proof}

\begin{remark}
It would be interesting to try to prove the full results of \cite{EsnaultViehwegSurUneMinoration} in characteristic $p > 0$.  This does not seem to follow from \cite{HochsterHunekeComparisonOfSymbolic}.  Indeed perhaps it is possible to simply mimic Esnault and Viehweg's proof (especially in light of Section \ref{sec.FPureCenters} below).
\end{remark}

\section{$F$-pure centers}
\label{sec.FPureCenters}

Suppose that $(X, \Delta \geq 0)$ is a pair such that $K_X + \Delta$ is $\bQ$-Cartier with index not divisible by $p > 0$.  Choose $e > 0$ such that $(1 - p^e)(K_X + \Delta)$ is Cartier.  Set $\sL_{e, \Delta} =  \O_X( (1 - p^e)(K_X + \Delta))$ and consider  $\phi_{\Delta} : F^e_* \sL_{e, \Delta} \to \O_X$ to be a map corresponding to $\Delta$ as in Section \ref{sec.Preliminaries}.   Recall the following definition.

\begin{definition} \cite{SchwedeCentersOfFPurity, SchwedeFAdjunction}
\label{def.FPureCenter}
With $(X, \Delta)$ as above, a subvariety $Z \subseteq X$ is called an \emph{$F$-pure center of $(X, \Delta)$} if $(X, \Delta)$ is sharply $F$-pure at the generic point of $Z$ and additionally if
\[
\phi_{\Delta}\big(F^e_* ( \sI_Z \cdot \sL_{e,\Delta}) \big) \subseteq \sI_Z
\]
where $\sI_Z$ is the ideal sheaf defining $Z \subseteq X$.
Thus we induce a non-zero map:
\[
\phi^Z_{\Delta} : F^e_* \big((\sL_{e,\Delta})|_Z\big)  \to \O_Z.
\]
If $Z$ is normal, then $\phi^Z_{\Delta}$ corresponds to an effective $\bQ$-divisor $\Delta_Z$ on $Z$ such that $(K_X + \Delta)|_Z \sim_{\bQ} K_Z + \Delta_Z$.
\end{definition}

\begin{remark}
Even more, if $Z$ is a union of $F$-pure centers, then it is easy to see that we still have a map
\[
\phi^Z_{\Delta} : F^e_* \big( (\sL_{e,\Delta})|_Z\big) \to \O_Z.
\]
\end{remark}

We now demonstrate how $F$-pure centers can be used to lift sections and to prove that adjoint linear systems are base-point-free along certain loci.

\begin{proposition}
\label{prop.SurjectiveOfSigmas}
Fix $X$ a normal projective variety and suppose that $(X, \Delta \geq 0)$ is a pair such that $K_X + \Delta$ is $\bQ$-Cartier with index not divisible by $p > 0$.  Suppose that $Z \subseteq X$ is any union of $F$-pure centers of $(X, \Delta)$.  Finally, suppose that $M$ is a Cartier divisor such that $M - K_X - \Delta$ is ample.  Then there is a natural surjective map:
\[
\FCohomology^0\big(X, \sigma(X, \Delta) \tensor \O_X(M) ) \to \FCohomology^0\big(Z, \sigma(Z, \phi_Z^{\Delta}) \tensor \O_Z(M)\big).
\]
\end{proposition}
\begin{proof}
Choose $e \gg 0$ such that $(1 - p^e)(K_X + \Delta)$ is Cartier and consider the diagram of short exact sequences:
\[
\xymatrix{
0 \ar[r] & F^e_* \left(\sI_Z \cdot \sL_{e, \Delta}\right) \ar[r] \ar[d] & F^e_* \sL_{e,\Delta} \ar[r] \ar[d] & F^e_* \left( \sL_{e,\Delta} \tensor \O_Z \right) \ar[d] \ar[r] & 0 \\
0 \ar[r] & \sI_Z \ar[r] & \O_X \ar[r] & \O_Z \ar[r] & 0
}
\]
Twisting by $M$ and taking cohomology, we obtain
\[
{\scriptsize
\xymatrix@C=10pt{
0 \ar[r] & H^0\big(X, F^e_* \left(\sI_Z \cdot \sL_{e, \Delta} \tensor \O_X(p^e M) \right)\big) \ar[r] \ar[d] & H^0\big(X, F^e_* (\sL_{e,\Delta} \tensor \O_X(p^e M)) \big) \ar[r]^-{\gamma} \ar[d]_{\alpha} & H^0\big(Z, F^e_* \left( \sL_{e,\Delta} \tensor \O_Z(p^eM) \right)\big) \ar[d]^{\beta} \ar[r] & 0 \\
0 \ar[r] & H^0\big(X, \sI_Z \cdot \O_X(M) \big) \ar[r] & H^0\big(X, \O_X(M)\big) \ar[r]_-{\delta} & H^0\big(Z, \O_Z(M)\big)
}
}
\]
where the top sequence is exact since
\[
H^1\big(X, F^e_* \left(\sI_Z \cdot \sL_{e, \Delta} \tensor \O_X(p^e M) \right)\big) = H^1\big(X, F^e_* \left(\sI_Z \cdot \O_X(p^e M - (p^e - 1)(K_X + \Delta) ) \right)\big)
\]
which vanishes for $e \gg 0$ by Serre vanishing.

Now, since $e \gg 0$, the image of $\alpha$ is $\FCohomology^0\big(X, \sigma(X, \Delta) \tensor \O_X(M) \big)$ and similarly the image of $\beta$ is $\FCohomology^0\big(Z, \sigma(Z, \phi_Z^{\Delta}) \tensor \O_Z(M)\big)$.  Choose $x = \beta(y) \in \FCohomology^0\big(Z, \sigma(Z, \phi_Z^{\Delta}) \tensor \O_Z(M)\big)$.  Fix $w \in H^0\big(X, F^e_* (\sL_{e,\Delta} \tensor \O_X(p^e M) ) \big)$ such that $\gamma(w) = y$.  Thus $\delta(\alpha(w)) = x$ and so the claimed map is surjective.
\end{proof}

\begin{remark}
If in the previous Proposition $Z$ is a finite union of closed points then it is easy to see that $\FCohomology^0\big(Z, \sigma(Z, \phi_Z^{\Delta}) \tensor \O_Z(M)\big) \cong H^0(Z, \O_Z)$.  It follows that $\FCohomology^0\big(X, \sigma(X, \Delta) \tensor \O_X(M) ) \subseteq H^0\big(X, \O_X(M)\big)$ separates the points of $Z$.
\end{remark}

Now we state a corollary.  The author expects that parts of this have been known to experts for several years:
\begin{corollary}
\label{cor.NoBasePointsAlongZ}
Suppose that $(X, \Delta)$ as in Proposition \ref{prop.SurjectiveOfSigmas}.  Suppose that $Z \subseteq X$ is a normal $F$-pure center of $(X, \Delta)$ and also suppose that $M$ is a Cartier divisor on $X$ such that $M - K_X - \Delta$ is ample.  Write $(K_X + \Delta)|_Z \sim_{\bQ} K_Z + \Delta_Z$ with $\Delta_Z$ as in Definition \ref{def.FPureCenter} and suppose that $(Z, \Delta_Z)$ is sharply $F$-pure.  Finally suppose that one of the following holds:
\begin{itemize}
\item[(i)]  $\O_Z(M)$ is equal to $\O_Z((\dim Z) A + B)$ for some globally generated ample Cartier divisor $A$ on $Z$ and some Cartier divisor $B$ such that $B - K_Z - \Delta_Z$ is ample, or
\item[(ii)]  $Z$ is a minimal $F$-pure center and $\O_Z(M)$ is equal to $\O_Z((\dim Z) A + B)$ for some globally generated ample Cartier divisor $A$ on $Z$ and some Cartier divisor $B$ such that $B - K_Z - \Delta_Z$ is big and nef, or
\item[(iii)]  $Z$ is a smooth curve and $M|_Z - K_Z - \Delta_Z \sim_{\bQ} (M - K_X - \Delta)|_Z$ has degree $>1$, or
\item[(iv)]  $Z$ is a closed point.
\end{itemize}
Then $|M|$ has no base points along $Z$.
\end{corollary}
\begin{proof}
In each case, we utilize Proposition \ref{prop.SurjectiveOfSigmas}.  Therefore, it is sufficient to show that the linear system associated to $\FCohomology^0\big(Z, \sigma(Z, \phi_Z^{\Delta}) \tensor \O_Z(M)\big)$ is base-point-free.  Note that $\O_Z = \sigma(Z, \phi_Z^{\Delta})$ by the sharp $F$-purity hypothesis, \cite{SchwedeFAdjunction}.

In case (i), we simply apply Theorem \ref{thm.BaseGlobalGenerationForSigmaTau}.  For case (ii), we notice that $(Z, \Delta_Z)$ is strongly $F$-regular by \cite{SchwedeFAdjunction} and so we apply Corollary \ref{cor.FullGlobalGeneration}.
Case (iv) is trivial.  For case (iii), we apply a similar strategy as in Theorem \ref{thm.TauCohomologyForCurves}.  Set $L_Z = (1-p^e)(K_Z + \Delta_Z) \sim_{\bQ} (1-p^e)(K_X + \Delta)|_Z$.  We consider the following diagram for $e \gg 0$ and any closed point $Q \in Z$.
\[
{\small
\xymatrix@C=15pt{
H^0\big(Z, \O_Z(M|_Z - Q)\big) \ar[r] & H^0\big(Z, \O_Z(M|_Z)\big) \ar[r]^-{\alpha} & H^0(Q, \O_Q) \\
H^0\big(Z, \O_Z(p^e (M|_Z - Q) + L_Z) \big)   \ar[u] \ar[r] & H^0\big(Z, \O_Z(p^e M|_Z + L_Z) \big)  \ar[u]^{\phi} \ar[r]_-{\beta} & H^0(Q, \O_{p^e Q}) \ar[u]_{\psi} \ar[r] & 0
}
}
\]
where the bottom row is exact since
\[
\begin{array}{rl}
& H^1\big(Z, \O_Z(p^e (M|_Z - Q) + L_Z) \big) \\
= & H^0\big(Z, \omega_Z(-p^e (M|_Z - Q) - L_Z) \big) \\
= & H^0\big(Z, \O_Z(K_Z - M|_Z + Q - (p^e - 1)(M|_Z - K_Z - \Delta_Z - Q)\big)\\
= & 0
\end{array}
\]
and the last equality follows since $e \gg 0$ and $\deg M|_Z - K_Z - \Delta_Z - Q > 0$.  Therefore $\beta$ is surjective and certainly $\psi$ is as well.  Thus $\alpha$ is surjective and the proof is complete.
\end{proof}
\begin{remark}
 One can apply the method of the previous theorem even without the assumption that the $F$-pure center $Z$ is normal.  Indeed, one only needs that $(Z, \phi_Z)$ is sharply $F$-pure (meaning that $\phi_Z$ is surjective) and that $S^0(Z, \sigma(Z, \phi_Z) \tensor \O_X(M))$ globally generates $\sigma(Z, \phi_Z) \tensor \O_X(M)$.
\end{remark}

\section{Further comments}
\label{sec.Comments}

Indeed, many of the results of this paper can easily be generalized to the context of $F$-pure Cartier modules in the sense of \cite{BlickleBoeckleCartierModulesFiniteness, BlickleTestIdealsViaAlgebras}.  On a scheme $X$, an \emph{$F$-pure Cartier-module} is a coherent $\O_X$-module $\sF$ together with a given surjective map
\[
\phi : F^e_* (\sF \tensor \sL) \to \sF
\]
for some choice of invertible sheaf $\sL$.  We note that the Cartier-modules introduced in \cite{BlickleBoeckleCartierModulesFiniteness} do not include the line bundle $\sL$ (or rather, always assumed $\sL = \O_X$), but the authors were certainly aware of this generalization.

Let us give an example of this.  Suppose that $X$ is any $d$-dimensional $F$-injective variety over a field.  Since $X$ is in particular S1, it follows that $\dim \Supp \myH^i(\omega_X^{\mydot}) \leq d + i - 1$ for $-d < i < 0$, see \cite{BlickleSchwedeTuckerTestAlterations}.  Now, notice that the natural map $F^e_* \myH^i(\omega_X^{\mydot}) \to \myH^i(\omega_X^{\mydot})$ is surjective since $X$ is $F$-injective.  Therefore, using the same argument as in \cite{KeelerFujita} or Theorem \ref{thm.BaseGlobalGenerationForSigmaTau}, we obtain the following:

\begin{proposition}
Suppose that $X$ is a projective $F$-injective $d$-dimensional variety.  Further suppose that $\sM$ is a globally generated ample line bundle and $\sL$ is any other ample line bundle.  Then the sheaf
\[
\myH^{i}(\omega_X^{\mydot}) \tensor \sM^{d + i - 1} \tensor \sL
\]
is globally generated for $-d < i < 0$.
\end{proposition}
\begin{proof}
Choosing some $e \gg 0$, we know that $\sF := F^e_* \big( \myH^i(\omega_X^{\mydot}) \tensor \sM^{p^e(d + i - 1)} \tensor \sL^{p^e} \big)$ satisfies the hypothesis of Theorem \ref{thm.Regularity} (in other words, it is $0$-regular with respect to $\sM$).  Thus $\sF$ is globally generated as an $\O_X$-module.  Thus the quotient $\myH^i(\omega_X^{\mydot}) \tensor \sM^{d + i - 1} \tensor \sL$ is also globally generated, proving the theorem.
\end{proof}



\def\cprime{$'$} \def\cprime{$'$}
  \def\cfudot#1{\ifmmode\setbox7\hbox{$\accent"5E#1$}\else
  \setbox7\hbox{\accent"5E#1}\penalty 10000\relax\fi\raise 1\ht7
  \hbox{\raise.1ex\hbox to 1\wd7{\hss.\hss}}\penalty 10000 \hskip-1\wd7\penalty
  10000\box7}
\providecommand{\bysame}{\leavevmode\hbox to3em{\hrulefill}\thinspace}
\providecommand{\MR}{\relax\ifhmode\unskip\space\fi MR}
\providecommand{\MRhref}[2]{%
  \href{http://www.ams.org/mathscinet-getitem?mr=#1}{#2}
}
\providecommand{\href}[2]{#2}

\end{document}